\documentclass[12pt,a4paper]{amsart}
\usepackage{a4wide} 

\usepackage{pdfpages}
 
 \usepackage{amsmath, amssymb, amsfonts, enumerate}
\usepackage[all]{xy}
\usepackage{amscd}
\usepackage{comment}
\usepackage{mathtools}
\usepackage{tikz} 
\usepackage{tikz-cd}

\usepackage{url}
\usepackage{hyperref}

\newtheorem{theoremletter}{Theorem}

\newtheorem{corollaryletter}{Corollary}

\makeatletter
\newtheorem*{rep@theorem}{\rep@title}
\newcommand{\newreptheorem}[2]{%
\newenvironment{rep#1}[1]{%
 \def\rep@title{#2 \ref{##1}}%
 \begin{rep@theorem}}%
 {\end{rep@theorem}}}
\makeatother

\newtheorem{theorem}{Theorem}[section]
\newtheorem*{theorem*}{Theorem}
 \newtheorem*{conjecture*}{Conjecture}
  
  \newtheorem*{corollary*}{Corollary}
  
  \newtheorem{lemma}[theorem]{Lemma}
\newtheorem{corollary}[theorem]{Corollary}
\newtheorem{proposition}[theorem]{Proposition}

\newreptheorem{theorem}{Theorem}
 \newreptheorem{proposition}{Proposition}
 \newreptheorem{corollary}{Corollary} 
 \theoremstyle{definition}
 \newtheorem{definition}[theorem]{Definition} 

 \newtheorem{remark}[theorem]{Remark}

\numberwithin{equation}{section}
 
\newcommand {\N}{\mathbb{N}} 
\newcommand {\Z}{\mathbb{Z}}

\newcommand {\C}{\mathbb{C}}

\newcommand{\CC}{\mathcal{C}}
\newcommand{\DD}{\mathcal{D}}

\newcommand{\OO}{\mathcal{O}}

\newcommand{\XX}{\mathcal{X}}

\newcommand{\A}{\mathbb{A}} 
\newcommand{\Proj}{\mathbb{P}}

 \newcommand{\Set}{\mathrm{Set}}

\newcommand{\Div}{\mathrm{Div}}

\newcommand{\Sch}{\mathrm{Sch}}

\DeclareMathOperator{\Mor}{Mor}
\DeclareMathOperator{\Hom}{Hom}

\DeclareMathOperator{\im}{Im}

\DeclareMathOperator{\Id}{Id}

\DeclareMathOperator{\supp}{supp}
\DeclareMathOperator{\Spec}{Spec}

\DeclareMathOperator{\pr}{pr}

\DeclareMathOperator{\mult}{mult} % mutiplicity
\DeclareMathOperator{\Lie}{Lie}

\DeclareMathOperator{\Hilb}{Hilb} 
\DeclareMathOperator{\Tr}{Tr} %Trace

\DeclareMathOperator{\val}{val}
\DeclareMathOperator{\Ass}{Ass}

\usepackage{setspace}
\begin{document}	
\title[Finiteness and uniformity of integral sections in some abelian fibrations]
{Finiteness criteria and uniformity of integral sections in some families of abelian varieties} 
\author[Xuan-Kien Phung]{Xuan Kien Phung}
\address{Universit\'e de Strasbourg, CNRS, IRMA UMR 7501, F-67000 Strasbourg, France}
\email{phung@math.unistra.fr}
\subjclass[2010]{}

\begin{abstract}   
Let $A$ be abelian variety over the function field $K$ of a compact Riemann surface $B$.  
Fix a model $f \colon \mathcal{A} \to B$ of $A/K$ and a certain 
effective horizontal divisor $\DD \subset \mathcal{A}$. 
We give a sufficient condition on the divisor $\DD$ 
for the finiteness of the set of $(S, \DD)$-integral sections 
for every finite subset $S \subset B$. 
These integral sections $\sigma$ correspond to rational points in $A(K)$ which satisfy   
the geometric condition $f ( \sigma(B) \cap \DD)\subset S$. 
This notion is the geometric variant of integral solutions of a system of \emph{Diophantine equations}. 
When $\mathcal{A}= A_0 \times B$ 
for some complex abelian variety $A_0$, 
we also give a certain uniform bound 
on the number of $(S, \DD)$-integral sections. 
For trivial families of abelian surfaces, 
a numerical criterion on $\DD$ for the finiteness of $(S, \DD)$-integral sections is obtained.   
\end{abstract}

\maketitle
 
\setcounter{tocdepth}{1}
\tableofcontents

%\section{Parshin-Arakelov theorem and integral points} 
%\label{c:parshin-arakelov-integral}

\section{Introduction}

\subsection*{Notations} 
We fix throughout a compact Riemann surface $B$ 
of genus $g$. 
Let $K= \C(B)$ be its function field. 
The letter $S$ will  denote a finite subset of $B$.  
For $m \in \N$, we denote $B^{(m)}$ the $m$-th symmetric power of $B$. 
Given a smooth separated scheme $X_K$ of 
finite type over $K$, 
\emph{a model of $X_K$ over $B$} is a locally finite type, proper and
flat scheme $\XX \to B$ endowed with an isomorphism $\XX_K \simeq X_K$. 
For every $P \in \XX_K(K)$, we denote usually $\sigma_P\in \XX(B)$ 
the induced section. The symbol $\#$ stands for cardinality. 
\par
\emph{A N\' eron model of $X_K$ over $B$}  
(cf. \cite[10.1, 1.2]{raynaud-neron-model-book})  
is a finite type smooth model $X$ of $X_K$ over $B$ 
satisfying the following universal property, 
called \emph{the N\' eron mapping property}: 
\begin{enumerate} [\rm ]
\item
For every smooth scheme morphism $Y \to B$, 
the following canonical map is a bijection  
 \begin{align}
\label{d:neron-mapping-property}
\Mor_B(Y, X) \to \Mor_K(Y_K, X_K). 
\end{align}
\end{enumerate}

\subsection{$(S, \DD)$-integral points}

In this article, we are interested in the finiteness and uniform finiteness 
properties of the set of $(S, \DD)$-integral sections in families of abelian varieties. 

\begin{definition} 
[$(S, \DD)$-integral point and section] 
\label{d:s-d-integral-point-geometric} 
Let $X_K$ be an abelian variety with a model (or a N\' eron model)  $f \colon  \XX \to B$. 
Let $S \subset B$ be a subset 
and let $\DD \subset \XX$ be an effective divisor. 
A section $\sigma \colon B \to \XX$ is $(S, \DD)$\emph{-integral}
 if it satisfies the set-theoretic condition: 
\begin{equation}
\label{e:s-d-integral-point-geometric}
f(\sigma(B) \cap \DD) \subset S. 
\end{equation}
For every $P \in X_K(K)$, let $\sigma_P \colon B \to X$ 
be the induced section. 
Then $P$ is said to be $(S, \DD)$\emph{-integral} 
if the section $\sigma_P$ is $(S, \DD)$-integral. 
\end{definition} 

The above definition of integral points is a geometric translation   
of the notion of integral solutions of systems of Diophantine equations. 
A  general definition of integral sections in an arbitrary fibration $\XX \to B$ is proposed and studied 
in \cite{phung-19-abelian}.

In \cite[Theorem 3.2]{parshin-90}, 
Parshin proved the following very general finiteness theorem 
on the set of $(S, \DD)$-integral points in abelian families:  

\begin{theorem}
[Parshin] 
\label{t:parshin-90}
Let $S \subset B$ be a finite subset. 
Let $A/K$ be an abelian variety and let $D$ be an effective reduced ample divisor on $A$. 
Let $\DD$ be the Zariski closure of $D$ in the N\' eron model $  \mathcal{A} \to B$ of $A$. 
Assume that $D$ does not contain any translate of nonzero abelian subvarieties of $A$. 
Then the number of $(S, \DD)$-integral points is finite modulo $\Tr_{K/\C}(A)(\C)$. 
\end{theorem}

By the Lang-N\' eron theorem (cf. \cite{lang-neron-theorem-paper}), 
$A(K)/\Tr_{K / \C}(A)(\C)$ is an abelian group of 
finite rank where $\Tr_{K / \C}(A)$ is the trace of $A$ (cf. \cite{chow-image-trace-1}, \cite{chow-image-trace-2}). 
When $A$ is simple, its trace is zero.  
When $A= A_0 \times K$ where $A_0$ is a complex abelian variety, 
we   have $\Tr_{K/ \C} (A)=A_0$.  
\par
The condition saying $D$ does not contain any translate of nonzero abelian subvarieties of $A$ is
always verified if $A$ is a simple abelian variety, e.g., 
when $A$ is an elliptic curve.  
\par
When $A/K$ is a nonisotrivial hyperbolic curve, 
we even have in the spirit of the Uniformity conjecture (cf. \cite{capo-harris-mazur})  
a certain uniform Mordell conjecture  due 
to Caporaso (cf. \cite{capo-99},  \cite{mordell}) as a consequence of her uniform version (cf. \cite{capo-99}) 
of the celebrated Parshin-Arakelov theorem (cf. \cite{parshin-68}, \cite{arakelov-71}):

\begin{theorem}
[Caporaso] 
\label{t:mordell-uniform-caporaso} 
There exists a number $M(q,g,s) \in \N$ such that 
for every nonisotrivial minimal surface $X$ fibered over $B$ 
with smooth fibres of genus $q \geq 2$ outside 
a finite subset $S \subset B$ of cardinality $s$, we have 
$\# X_K(K) \leq M(q,g,s)$. 
\end{theorem}

The goal of the article is to give criteria for the divisor $\DD$ 
which imply the finiteness of $(S, \DD)$-integral sections 
for every finite subset $S \subset B$ 
(Theorem \ref{c:raynaud-parshin-buium}, 
Theorem \ref{c:raynaud-parshin-buium-strong}, 
Theorem \ref{t:uniform-abelian-trivial}.(iii)). 
If $\mathcal{A}= A_0 \times B$ 
where $A_0$ is a complex abelian variety, 
we obtain also some uniform results 
on the finiteness of $(S, \DD)$-integral points 
(Theorem \ref{t:uniform-abelian-trivial}, 
Theorem \ref{t:NW-generalized-intro}, Corollary  \ref{p:abelian-nontrivial-intro}).

%\section[A finiteness result on the set of integral points]{On the finiteness of the set of integral points in abelian varieties}  
%\label{s:raynaud-parshin-buium}

\subsection{Finiteness criteria in the general case} 
Let $A/K$ be an abelian variety with \emph{any} model $f \colon \mathcal{A} \to B$.  
Assume that $D \subset A$ is an effective ample divisor on $A$. 
Let $\DD $ be the Zariski closure of $D$ in $\mathcal{A}$. 
The first goal of the article is to give a general sufficient condition on $\DD$ for the finiteness of the set 
of $(S, \DD)$-integral sections for every finite subset $S \subset B$, 
even without taking modulo $\Tr_{K/ \C}(A)(\C)$.  
More specifically, we obtain (cf. Section \ref{s:criterion-general}): 

\begin{theoremletter}
\label{c:raynaud-parshin-buium}
Let the notations be as above.  
Suppose that $D$ does not contain any translate of nonzero abelian subvarieties of $A$. 
If moreover $D(K)=\varnothing$, 
then  for every finite subset $S \subset B$,  the set of $(S, \DD)$-integral sections of $\mathcal{A}$ is finite. 
\end{theoremletter}

The condition $D(K)= \varnothing$ 
should be interpreted as the divisor $D$ being in a very general position. %so that it does not contains any rational points. 
In fact, we obtain the following stronger useful statement (cf. Section \ref{s:criterion-general}):  

\begin{theoremletter}
\label{c:raynaud-parshin-buium-strong}
Let $A/K$ be an abelian variety with $D \subset A$ an effective ample divisor. 
Let $P\in A(K)$ be a rational point.  
Assume that for every $a \in \Tr_{K/ \C}(A)(\C)$, 
$P+a  \notin D$. 
Then for every finite subset $S \subset B$, 
the set $P+ \Tr_{K/ \C}(A)(\C)$ contains only finitely many $(S, \DD)$-integral points of $A$ 
with respect to {any} model $\mathcal{A} \to B$  of $A$ and with 
$\DD$ being the Zariski closure of $D$ in $\mathcal{A}$.    
\end{theoremletter} 

For each $a \in A(K)$ and $\sigma \in \mathcal{A}(B)$, the notation $a+ \sigma_P$ will stand for the 
section of $\mathcal{A} \to B$ induced by $a + P \in A(K)$.  
We can moreover prove that (cf. Section \ref{s:criterion-general}): 

\begin{theoremletter}
\label{t:raynaud-parshin-buium-strong-finite-map}
Let $A/K$ be an abelian variety with a N\' eron model $f \colon \mathcal{A} \to B$.  
Assume that $D \subset A$ is an effective ample divisor on $A$. 
Let $\DD $ be the Zariski closure of $D$ in $\mathcal{A}$. 
Let $P\in A(K)$ and let $\sigma_P \in \mathcal{A}(B)$ be the induced section. 
Assume $D \cap ( P + \Tr_{K/ \C}(A)(\C) ) =  \varnothing$.  
Then for $r\coloneqq \deg_B \sigma_P^*\DD$,  we have a finite morphism of complex schemes 
\begin{equation}
\pi  \colon \Tr_{K/ \C}(A) \to B^{(r)}, \quad  a \mapsto (a + \sigma_P)^*\DD.
\end{equation}
\end{theoremletter} 

The condition $D \cap ( P + \Tr_{K/ \C}(A)(\C) ) =  \varnothing$ for some $P \in A(K)$  is 
a reasonably mild condition on $D$. 
For example, it is always satisfied if $D$ does not contain any rational point in $A(K)$, i.e., 
if $D(K)= \varnothing$. 
Moreover, let $\Gamma= A(K) / \Tr_{K/ \C} (A)(\C)$ be the Mordell-Weil group of $A$. 
Then we can see $A(K)= \coprod_{[P_i] \in \Gamma} P_i + \Tr_{K/ \C}(A)(\C)$ as a countable disjoint 
union of copies of $\Tr_{K/ \C}(A)(\C)$. 
Then it suffices for $D$ to avoid one of these copies, says, $P+  \Tr_{K/ \C}(A)(\C)$, 
so that  $P+a  \notin D$ for every $a \in \Tr_{K/ \C}(A)(\C)$. 
\par
If $A$ and $D$ are defined over $\C$, 
the finiteness statement in Theorem \ref{c:raynaud-parshin-buium-strong} 
can be made uniform in terms of certain numerical invariants 
(cf. Theorem \ref{t:uniform-abelian-trivial}.(ii) below). 
\par
As an application, 
we obtain the following result (cf. Section \ref{s:criterion-general})
on the topology of the intersection locus of sections with a divisor 
in a family of abelian varieties. 

\begin{corollaryletter} 
\label{t:raynaud-parshin-buium-strong-dense}
Let the hypotheses be as in Theorem \ref{t:raynaud-parshin-buium-strong-finite-map}. 
Let $R \subset \Tr_{K/ \C}(A)(\C) \subset A(K)$ be a subset and let 
$I(R)\coloneqq \cup_{a \in R} f( (a+ \sigma_P)(B) \cap \DD)) \subset B$. 
The following hold:  
\begin{enumerate} [\rm (i)] 
\item
if $R$ is infinite, then $I(R)$ is Zariski dense, i.e., infinite, in $B$; 
\item
if $R$ is analytically dense in a complex algebraic curve $C \subset \Tr_{K /\C}(A)$, 
then $I(R)$ is analytically dense in $B$.  
\end{enumerate}
\end{corollaryletter}

In fact, the proof of Corollary \ref{t:raynaud-parshin-buium-strong-dense} 
in Section \ref{s:criterion-general} shows a slightly stronger property than stated above.   
\par
The motivation for Corollary \ref{t:raynaud-parshin-buium-strong-dense} is the following. 
Let the notations be as in Corollary \ref{t:raynaud-parshin-buium-strong-dense} but 
assume  $\Tr_{K / \C}(A)=0$ and 
$D$ not containing any translate of nonzero abelian subvarieties 
instead of the hypothesis $D \cap ( P + \Tr_{K/ \C}(A)(\C) ) =  \varnothing$ for some $P \in A(K)$.  
Let $R' \subset A(K)$ be any infinite subset. 
It is shown in \cite[Theorem B]{phung-19-abelian} 
that  $I(R')\coloneqq \cup_{P \in R'} f( \sigma_P(B) \cap \DD) \subset B$ 
has uncountably many 
limit points in $B$ and cannot be even contained in 
any finite union of disjoint small analytic discs in $B$. 
\par 
Suppose now that $B$ is a smooth projective curve defined over a number field $k$. 
Let $K= k(B)$.  
Let $f \colon \mathcal{A} \to B$ be a nonisotrivial elliptic surface. 
Assume that $\DD$ is the zero section of $f$ and $R'= \{nQ \colon n \in \N^* \} \subset \mathcal{A}_K(K)$ for 
some non torsion point $Q \in \mathcal{A}_K(K)$.  
Let $I(R') \coloneqq \cup_{P \in R'} f( \sigma_P(B\left(\bar{k}\right) \cap \DD)) \subset B\left(\bar{k} \right)$.  
Then $I(R')$ is analytically dense in $B(\C)$ (cf. \cite[Notes to chapter 3]{zannier-unlikely-intersection}). 
In fact,  results in \cite{demarco-mavraki} show that $I(R')$ 
is even equidistributed in $B(\C)$ with respected to a certain Galois-invariant measure.

% since the transcendental dimension of $ \Tr_{K/ \C}(A)(\C)$ over $\C$ is zero, while the transcendental dimension  
%For example, if $A/K$ is an abelian surface with $\dim \Tr_{K/ \C}(A)=1$, i.e., 
%$\Tr_{K/ \C}(A)$ is a complex elliptic curve, and if $\Tr_{K/ \C}(A) \not \subset D$, 
%then $P+ a  \notin D$ for every $a \in \Tr_{K/ \C}(A)(\C)$ and every $P \in A(K) \setminus \Tr_{K/ \C}(A)(\C)$. 
  
%It turns out that Theorem \ref{c:raynaud-parshin-buium-strong} can be used to 

\subsection{Effective criterion and uniform finiteness in trivial families} 

Our next motivation is the following uniform result of Noguchi-Winkelmann (cf. \cite{noguchi-winkelmann-04}) 
on the intersection multiplicities of curves with an ample divisor in a complex abelian variety. 

\begin{theorem}
[Noguchi-Winkelmann] 
\label{t:noguchi04}
There is a function $M_{NW} \colon  \N^3 \to \N$ satisfying the following property. 
Let $A$ be a complex abelian variety of dimension $n$, 
let $D$ be an ample effective divisor on $A$ with intersection number $D^n = d$. 
Let $f \colon B \to A$ be a morphism. 
Then either $f(B) \subset D$ or $\text{mult}_x f^*D \leq M_{NW}(g,n,d)$ for all $x \in B$.
\end{theorem}

In the other extreme situation where the abelian variety is absolutely traceless, 
we also have the following uniform bound of intersection multiplicities of Buium (cf. \cite{buium-94}).

\begin{theorem}
[Buium] 
\label{t:buium94}
Let $A/K$ be an abelian variety with $\Tr_{\bar{K} / \C}(A_{\bar{K}})=0$. 
Then for each rational function $f \in K(A)$, there exists a positive constant $m(A,K,f)$ such that 
for every $P \in A(K)$  where $f$ is defined and does not vanish, 
the zeroes and poles of $f(P) \in K^*$ have orders at most $m(A,K,f)$. 
\end{theorem}

No uniform bounds on the intersection multiplicities 
are known for general abelian varieties $A/K$ in the literature. 
However, certain related results are recently obtained in \cite{phung-19-abelian} 
for general families of abelian varieties and for generalized $(S, \DD)$-integral sections. 
\par 
The second goal of the article is to obtain the following uniform bound 
 on  the number of 
integral points in a constant family of abelian varieties. 
\par
Let $n, g, d, s \in \N$. 
Consider a finite subset $S$ of cardinality at most $s$ of 
a compact Riemann surface $B$ of genus $g$. 
Let $A$ be a complex abelian variety of dimension $n$ 
with an effective ample divisor $D$ of degree $D^n = d$. 
 Denote by $W$ the set of nonconstant algebraic morphisms 
$f \colon (B \setminus S) \to (A \setminus D)$. 
Using Theorem \ref{t:noguchi04} 
and Theorem \ref{c:raynaud-parshin-buium-strong}, 
the following result is proved (cf. Section \ref{s:noguchi-uniform}): 

\begin{theoremletter}
\label{t:uniform-abelian-trivial}
 There exists a function $N \colon \N^4 \to \N$ such that: 
\begin{enumerate} [\rm (i)]
\item 
either $W$ is infinite or $\# W \leq N(g,s,n,d)$; 
\item
$\# \{ f \in W \colon  a+ \im f \nsubseteq D,\, \forall a \in A \} \leq N(g,s,n,d)$; 
\item
if $n=2$, $d> 2g-2$ and $D$ is integral then  
$\# W \leq N(g,s,n,d)$. 
\end{enumerate}
\end{theoremletter}

\begin{remark} 
Let $\mathcal{A}= A \times B$. Then the set $W$ of 
nonconstant algebraic morphisms 
$f \colon (B\setminus S) \to (A \setminus D)$ in Theorem 
\ref{t:uniform-abelian-trivial} is exactly the set 
nonconstant $(S, \DD)$-integral sections of $\mathcal{A}  \to B$ where $\DD= D \times B$. 
\end{remark}

The case of an arbitrary effective divisor $\DD \subset A \times B$ is also treated 
with the tools of jet-differentials as in the proof of Theorem \ref{t:noguchi04} as follows  
(cf. Theorem \ref{t:NW-generalized}). 
 
\begin{theoremletter}
\label{t:NW-generalized-intro}
Let $A$ be a complex abelian variety. 
Let $\DD$ be an integral divisor in $A \times B$. 
There exists a number $M >0$ satisfying the following property.  
For every morphism $f \colon B \to A$ such that $(f \times \Id_B)(B) \nsubseteq \DD$, 
we have an estimation 
\begin{equation}
 \mult_x (f \times \Id)^* \DD \leq M \quad \text{ for all } x \in B.
 \end{equation} 
\end{theoremletter}

As an application, 
we establish in Section \ref{s:semi-bound-integral-point} the following semi-effective bound on the number of integral points 
of \emph{bounded denominators} (see also  \cite{hindry-silverman-88},  
\cite[Corollary 1.8]{phung-19-parshin-integral} for the case of elliptic curves and \cite{phung-19-abelian}, \cite{phung-19-elliptic} 
for generalized integral points):

\begin{corollaryletter}
\label{p:abelian-nontrivial-intro}
Let $A$ be a complex abelian 
variety of dimension $n$. 
Let $\DD \subset A \times B$ be an effective divisor such that $\DD_K$ is ample.  
For each integer $s \geq 1$, 
let $W(s, \DD)$ be the set of morphisms $f \colon B \to A$ such that 
$\# (f \times \Id_B)(B) \cap \DD \leq s$. 
Then there exists a number $H>0$ such that for any $s \geq 1$ we have a semi-effective bound 
$$ \# W(s,\DD) \text{ modulo } A(\C) \leq (2\sqrt{sH}+1)^{4gn}.$$
\end{corollaryletter}

As an application, we obtain a certain generic emptyness result 
on the set of $(S, \DD)$-integral points in a trivial family of  abelian varieties 
in Section \ref{s:generic-empty-S-abelian-trivial}  
  (cf. Corollary \ref{t:generic-empty-S-abelian-trivial}).

\section{Preliminary lemmata on N\' eron models} 

We begin with the following lemma: 

\begin{lemma}
\label{l:trace-injective-general}
Let $A/K$ be an abelian variety with  a model $f \colon \mathcal{A} \to B$. 
Then there exists a nonempty Zariski open subset $U \subset B$ such that 
for every $b \in U$, the evaluation map 
$$
\val_b \colon \Tr_{K/ \C} A(\C) \to \mathcal{A}_b, \quad P \mapsto \sigma_P(b),
$$
is an injective morphism of varieties, where $\sigma_P \colon B \to \mathcal{A}$ is the corresponding section of $P$. 
\end{lemma}

\begin{proof}
Let $f_{\textit{N\' eron}} \colon \mathcal{A}^{\textit{N\' eron}} \to B$ be the minimal N\' eron model of 
$A/K$ (\cite[Theorem 1.4.3]{raynaud-neron-model-book}). 
By definition of the trace, the canonical map 
$\iota_K \colon \Tr_{K / \C} (A) \otimes_\C K \to A$ is a closed immersion 
homomorphism of $K$-abelian varieties. 
Consider the smooth $B$-abelian scheme $\Tr_{K/ \C} (A) \times B$. 
Then by the N\' eron mapping property (cf. \eqref{d:neron-mapping-property}), 
the map $\iota_K$ extends to a unique $B$-morphism 
$\iota \colon \Tr_{K/ \C} (A) \times B \to \mathcal{A}^{\textit{N\' eron}}$. 
By \cite[Proposition 9.6.1.(ix)]{ega-4-3}, there exist a nonempty Zariski open subset $U_1$ such that 
for every $b \in U_1$, the induced base change map 
$$
\iota_b \colon (\Tr_{K / \C} (A) \times B ) \otimes \kappa(b) \to \mathcal{A}^{\textit{N\' eron}}_b
$$
is also a closed immersion and in particular it is injective. 
\par
Similarly, the identity $K$-map $A \to A$ 
extends to a unique $B$-morphism $h \colon \mathcal{A} \to \mathcal{A}^{\textit{N\' eron}}$ 
by the N\' eron mapping property. 
 It is clear that there exists a nonempty Zariski open subset $U_2 \subset B$ above which 
$h$ is a fibrewise isomorphism (again by \cite[Proposition 9.6.1.(ix)]{ega-4-3}). 
Let $U = U_1 \cap U_2$. 
 We claim that the map  
$h \circ \val_b$ is exactly the reduction map 
$\iota_b$ for every $b \in U$. 
Indeed, let $P \in \Tr_{K / \C}(A)(\C)$ and let $\tau_P \colon B \to \mathcal{A}^{\textit{N\' eron}}$ be the corresponding 
section. % (which exists and is unique by the N\' eron mapping property). 
Then we must have $h \circ \sigma_P= \tau_P$ by the N\' eron mapping property and this proves $h \circ \val_b= \iota_b$. 
Therefore, the conclusion follows  
since $\val_b$ is a closed immersion and
 the fact that 
$h_b \colon \mathcal{A}_b \to \mathcal{A}^{\textit{N\' eron}}_b$ 
is an isomorphism for every $b \in U$. 
\end{proof}

 Recall the following universal property of the symmetric powers of a curve 
which is needed in the proof of Lemma \ref{l:degree-translate-abelian}.  
 
\begin{proposition}
\label{p:universal-symmetric-power} 
Let $B$ be a complete smooth curve over a field $k$ and 
$T$ a $k$-scheme. 
Let $\Div^r_B(T)$ be the set 
of relative effective Cartier divisors on $C\times T  \to T$ of degree $r$. 
Then $\Div^r_B \colon \Sch_k \to \Set$ is a functor represented by $B^{(r)}$. 
Moreover, for any relative effective divisor $D$ on $B \times T /T$ of degree $r$, 
there exists a unique morphism $\varphi \colon T \to B^{(r)} $ such that $D= (\Id \times \varphi)^* D_{can}$ 
where 
$D_{can}= S_r \backslash (\sum_{i=1}^r \Delta_i)$ 
with $\Delta_i \subset B \times B^r$ the image of the section 
$B^r \to B \times B ^r$ given by $(b_1, \cdots, b_r) \mapsto (b_i, b_1, \cdots, b_r)$.  
\end{proposition}

\begin{proof}
See \cite[Theorem 3.13]{milne-JVs}. 
\end{proof}

\begin{remark}
\label{rem-pullback-divisor}
Let $f \colon X \to Y$ be a morphism of schemes. 
Let $D$ be an effective Cartier divisor on $Y$. 
Then the pullback divisor $f^*D$ is defined if and only if $f(\Ass (X)) \cap \supp (D)= \varnothing$, 
if and only if $f^{-1}(D)$ is an effective divisor. 
In this case, we have $\supp f^*D= f^{-1} (D)$. 
Here, $\Ass(X)$ denotes the set of associated points of $X$ (cf. \cite[Definition 7.1.1]{liu-alg-geom}). 
\end{remark}

\begin{lemma}
\label{l:degree-translate-abelian} 
Let the notations be as in Theorem \ref{c:raynaud-parshin-buium} and assume furthermore that $\mathcal{A}$ is a N\' eron model of $A$ over $B$.  
Let $\sigma \colon B \to \mathcal{A}$ be a section and let $r\coloneqq \deg_B \sigma^*\DD$. 
Then we have 
$\deg_B(a + \sigma)^*\DD = r$ for every $a \in \Tr_{K / \C} (A)(\C)$. 
Moreover, we have a well-defined morphism of complex schemes  
$\pi  \colon \Tr_{K/ \C}(A) \to B^{(r)}$, $a \mapsto \sigma_a^*\DD= (a + \sigma)^*\DD$. 
\end{lemma} 

In the above lemma, the notation $a+ \sigma$ means   the 
section of $\mathcal{A} \to B$ associated 
to the rational point $a + P_\sigma \in A(K)$ where 
$P_\sigma \in A(K)$ is the rational point corresponding to $\sigma \in \mathcal{A}(B)$.

\begin{proof}[Proof of Lemma \ref{l:degree-translate-abelian}]
As in Lemma \ref{l:trace-injective-general}, we have a $B$-morphism 
$\lambda \colon \Tr_{K/ \C} (A) \times B \to \mathcal{A}$ which extends  
the closed immersion $\tau_\sigma \circ \iota_K \colon \Tr_{K / \C} (A) \otimes_\C K \to A$. 
Here, $\tau_\sigma \colon A \to A$ denotes the translation by the rational point $P_\sigma \in A(K)$ 
associated to $\sigma$. 
Hence, we obtain a section 
$$
\Sigma= (\lambda, \pi_2) \colon B \times  \Tr_{K/ \C} (A) \to \mathcal{A} \times \Tr_{K/ \C} (A) 
$$
of the morphism $\pi=f \times \Id \colon \mathcal{A} \times \Tr_{K/ \C} (A) \to B \times \Tr_{K/ \C} (A)$. 
Here $\pi_2 \colon B \times \Tr_{K/ \C} (A) \to \Tr_{K/ \C} (A)$ denotes the second projection. 
\par
By the assumption $D(K)= \varnothing$ and by the dimensional reason, 
the divisor $\DD  \times \Tr_{K/ \C} (A)$ of $\mathcal{A} \times \Tr_{K/ \C} (A)$ 
is not contained in the image of $\Sigma$. 
We can thus define the effective Cartier divisor 
$R \coloneqq \Sigma^* (\DD \times \Tr_{K/ \C} (A))$ on $B \times \Tr_{K/ \C} (A)$ 
(cf. Remark \ref{rem-pullback-divisor}). 
We claim that 
$R$ is a relative Cartier divisor of $B \times \Tr_{K/ \C} (A) \to \Tr_{K/ \C} (A)$ with  
$R_a= (a+ \sigma)^* \DD$ for every $a \in \Tr_{K/ \C}(A)(\C)$. 
Indeed, fix  $a \in \Tr_{K/ \C}(A)(\C)$ and consider the following commutative diagram:  
 \[
\label{d-degree-translate-abelian}
\begin{tikzcd}
\mathcal{A} \times \{a\} \arrow[r, hook, "i_a"]   \arrow[d, "f_a", swap] 
 &  \mathcal{A} \times \Tr_{K / \C}(A)  \arrow[d, "\pi", swap] 
 & \DD \times \Tr_{K / \C}(A)    \arrow[l, hook]  \\ 
B \times \{a \}  \arrow[r, hook, "j_a"]   \arrow[u, bend right=45, swap, "\Sigma_a"] 
 & B \times \Tr_{K/ \C}(A)  \arrow[u, bend right=45, swap, "\Sigma"] 
& 
\end{tikzcd}
\]
Since $D(K) = \varnothing$, we have 
$\im (a+ \sigma) \not \subset \DD$ and $\im (j_a) \not \subset R$. 
As $\Sigma_a= a + \sigma$ and $\Sigma_a \circ i_a= j_a \circ \Sigma$, 
we deduce the following equality of effective Cartier divisors: 
$$
R_a
=  j_a^* \left( \Sigma^*  \left( \DD \times \Tr_{K / \C}(A) \right) \right) 
= (a + \sigma)^* \left(i_a ^* \left( \DD \times \Tr_{K / \C}(A) \right) \right)
= (a + \sigma)^* \DD. 
$$ 
Therefore, every fibre $R_a$ is an effective Cartier divisors. 
$R$ is clearly a locally principal closed subscheme of $B \times \Tr_{K/ \C}(A)$ as it is an effective Cartier divisor.  
We deduce from \cite[Lemma 062Y]{stack-project} %Lemma 30.18.9 in Stack Projects 
that $R$ is a relative effective Cartier divisor 
of $B \times \Tr_{K / \C} (A)/ \Tr_{K/ \C}(A)$, i.e., $R$ as a closed subscheme is flat over $\Tr_{K/ \C}(A)$. 
In particular, the curves $(a + \sigma)(B) \subset \mathcal{A}$ are algebraically equivalent. 
It follows that 
$\deg_B(a + \sigma)^*\DD = r$ for every $a \in \Tr_{K / \C} (A)(\C)$. 
By the universal property of the symmetric power $B^{(r)}=\Hilb_B^r$ (cf. Proposition \ref{p:universal-symmetric-power}), 
we conclude that 
$$
\pi  \colon \Tr_{K/ \C}(A) \to B^{(r)}, \quad a \mapsto R_a=  (a + \sigma)^*\DD 
$$
is indeed a morphism of complex schemes.  
\end{proof}

\section{Finiteness criteria in the general case}
\label{s:criterion-general}
We are in position to prove Theorem \ref{c:raynaud-parshin-buium}.

\begin{proof}[Proof of Theorem \ref{c:raynaud-parshin-buium}] 
By Theorem   \ref{t:parshin-90}, 
the set of $(S,\DD)$-integral points is finite modulo $\Tr_{K/\C} (A)$. 
Suppose now that $D(K)=\varnothing$, i.e., $\DD$ does not contain any section of $f \colon \mathcal{A} \to B$. 
We first consider the case when $\mathcal{A}= \mathcal{A}^{\textit{N\' eron}}$ is a N\' eron model of $A_K$ over $B$. 
Since the divisor $\DD$ is generically ample, it is ample over a dense open affine subset $U \subset B \setminus S$. 
In particular,  
$\DD\vert_U$ is an ample divisor of $\mathcal{A}\vert_U$ so that  
$(\mathcal{A}\setminus \DD)\vert_U$ is affine as $\mathcal{A}\vert_U$ is proper over $U$. 
% (cf. Lemma \ref{l:ample-affine}).    
Let $\sigma \colon B \to \mathcal{A}$ be a section and let $r\coloneqq \deg_B \sigma^*\OO(\DD)$. 
Let $P \in A(K)$ be the corresponding rational points of $\sigma$. 
Then the condition $D(K) =\varnothing$ implies by Lemma \ref{l:degree-translate-abelian} that  
% as in the proof of Theorem \ref{p:quasi-finite-modA}, 
we have a well-defined morphism of complex schemes: 
\begin{equation*}
\pi  \colon \Tr_{K/\C}(A) \to B^{(r)}, \quad a \mapsto \sigma_a^*\DD= (a + \sigma)^*\DD. 
\end{equation*}
\par
Let $S_i$ be an effective divisor of $B$ of degree $r$ such that $\supp S_i \subset S$.  
There are clearly only finitely many such $S_i$. 
Let $[S_i] \in B^{(r)}$ be the image of $S$ in $B^{(r)}$. 
Let $W_i = \pi^{-1} ([S_i])$ then $W_i$ is a Zariski closed subset of $\Tr_{K/\C}(A)$. 
Remark that the union of the sets $P+W_i$ is exactly the set of \emph{all} $(S, \DD)$-integral points   
of $A$ in the translation class modulo $\Tr_{K/ \C}(A)(\C)$ of $P$ 
with respect to the model $\mathcal{A} \to B$ and the divisor $\DD$. 
\par  
Since $\Tr_{K/ \C}(A)$ is an abelian variety, it is proper thus so is $W_i$. 
Let $x_0 \subset U$ be a general point. 
Since $U \subset B \setminus S$, we can  
consider the morphism  (cf. Lemma \ref{l:trace-injective-general}) 
$$\Sigma_i \colon W_i \to \mathcal{A}_{x_0} \setminus \DD_{x_0} \subset (\mathcal{A} \setminus \DD)\vert_U, 
\quad a \mapsto a(x_0) + \sigma(x_0).$$ 
Since $W_i$ is proper, $\im \Sigma_i$ is also proper. 
Moreover, 
since $\mathcal{A}_{x_0} \setminus \DD_{x_0}$  is affine, we deduce that 
$\im \Sigma_i$ is a proper affine scheme. 
Hence $\im \Sigma_i$ is finite. 
It follows immediately that $W_i$ is also finite since 
$\Sigma_i$ is an injective morphism for a general choice of $x_0$ (cf. Lemma \ref{l:trace-injective-general}).   
 \par
 Thus, the number of $(S, \DD)$-integral points in each class modulo $\Tr_{K/ \C}(A)(\C)$ 
is finite. 
On the other hand, as mentioned at the beginning of the proof,  
the set of $(S,\DD)$-integral points is finite modulo $\Tr_{K/ \C}(A)(\C)$. 
Therefore, we conclude that the number of $(S,\DD)$-integral points is finite. 
\par
Now consider the general case when $\mathcal{A}$ is not necessarily a N\' eron model 
of $A$ over $B$. 
Let $f' \colon \mathcal{A}' \to B$ be a N\' eron model of $A$ over $B$. 
By the N\' eron mapping property, there exists a $B$-morphism 
$h \colon \mathcal{A} \to \mathcal{A'}$ extending the identity $K$-map 
$\Id \colon A_K \to A_K$. 
Let $T \subset B$ be the finite subset above which the fibres of $f \colon \mathcal{A} \to B$ are not smooth. 
Then $\mathcal{A}_{B\setminus T}$ and 
$\mathcal{A}'_{B\setminus T}$ are N\' eron models of $A_K$ over 
$B \setminus T$ because they are abelian schemes over $B\setminus T$ 
(cf.   \cite[Proposition 1.2.8]{raynaud-neron-model-book}). 
Again by the N\' eron mapping property (over $B \setminus T$), we find that $h\vert_{(B \setminus T)}$ 
must be a $(B \setminus T)$-isomorphism. 
\par
On the other hand, for every section $\sigma \colon B \to \mathcal{A}$ associated to a 
rational point $P \in A(K)$, 
$h \circ \sigma \colon B \to \mathcal{A}'$ is exactly the corresponding section of $P$ 
in $\mathcal{A}'$. 
Hence, it follows from Definition \ref{d:s-d-integral-point-geometric} that 
every $(S,D)$-integral point of $\mathcal{A}$ is an $(S\cup T, \DD')$-integral point of $\mathcal{A}'$ where $\DD'$ is the Zariski closure of $D$ in $\mathcal{A}'$. 
Since $S\cup T$ is finite, Theorem \ref{c:raynaud-parshin-buium}.(i) now follows from the corresponding results for $(S\cup T, \DD')$-integral points in $\mathcal{A}'$.   
\end{proof}

In fact, the proof of Theorem \ref{c:raynaud-parshin-buium-strong} is already contained in the above arguments. 

\begin{proof}[Proof of Theorem \ref{c:raynaud-parshin-buium-strong}] 
Let $\sigma_P \colon B \to \mathcal{A}$ be a section. 
It is not hard to see that in the course of the proofs of Lemma \ref{l:degree-translate-abelian} 
and of Theorem \ref{c:raynaud-parshin-buium}, 
we only need to use the condition $(a + \sigma_P)(B) \not \subset \DD$ for every $a \in \Tr_{K/ \C}(A)(\C)$  
to show that the number of $(S, \DD)$-integral sections 
in the translation class of $\sigma_P$ by the trace is finite.  
\par 
On the other hand, the condition $(a + \sigma_P)(B) \not \subset \DD$ for every $a \in \Tr_{K/ \C}(A)(\C)$ 
is equivalent to $a + P  \notin D$ for every $a \in \Tr_{K/ \C}(A)(\C)$. 
Hence, the choice of the models $\mathcal{A} \to B$ and $\DD$ of $A$, $D$ does not matter.     
Therefore,  Theorem \ref{c:raynaud-parshin-buium-strong} is proved. 
\end{proof}

Similarly, we can deduce immediately Theorem \ref{t:raynaud-parshin-buium-strong-finite-map}. 

\begin{proof}[Proof of Theorem \ref{t:raynaud-parshin-buium-strong-finite-map}] 
Lemma \ref{l:degree-translate-abelian} and the first part in the proof of Theorem \ref{c:raynaud-parshin-buium} 
actually show that the morphism 
\begin{equation}
\pi  \colon \Tr_{K/ \C}(A) \to B^{(r)}, \quad  a \mapsto (a + \sigma_P)^*\DD 
\end{equation} 
is well-defined and is quasi-finite where $r= \deg_B \sigma_P^* \OO(\DD) \in \N$. 
Since  $\Tr_{K/ \C}(A)$ and $B^{(r)}$ are proper varieties, $\pi$ is in fact 
a proper morphism. 
Therefore, 
$\pi$ is a quasi-finite proper morphism of varieties. 
By \cite[Th\' eor\` eme 8.11.1]{ega-4-3}, it follows that $\pi$ is indeed a finite morphism as desired. 
\end{proof}

The proof of Corollary \ref{t:raynaud-parshin-buium-strong-dense} of Theorem \ref{t:raynaud-parshin-buium-strong-finite-map} 
can also be given now as follows. 

\begin{proof}[Proof of Corollary \ref{t:raynaud-parshin-buium-strong-dense}] 
Let $r= \deg_B \sigma_P^* \OO(\DD) \in \N$, then the map  
\begin{equation}
\label{e:pi-finite-corollary-density}
\pi  \colon \Tr_{K/ \C}(A) \to B^{(r)}, \quad  a \mapsto (a + \sigma_P)^*\DD 
\end{equation} 
is a finite morphism by Theorem  \ref{t:raynaud-parshin-buium-strong-finite-map}. 
Suppose first that the set $R \subset  \Tr_{K/ \C}(A)(\C)$ is infinite. 
It follows that 
the image $\pi(R) \subset B^{(r)}$ is also infinite. 
Consider the canonical finite morphism $q_r \colon B^r \to B^{(r)}$. 
Then the preimage $E \coloneqq q_r^{-1}(\pi(R)) \subset B^r$ is also infinite. 
Let $p_i \colon B^r \to B$ be the $i$-th projection for $i=1, \dots, r$. 
It is a straightforward verification from the definition of the map $\pi$ in 
\eqref{e:pi-finite-corollary-density} that 
$$
I(R)\coloneqq \cup_{a \in R} f( (a+ \sigma_P)(B) \cap \DD)) = \cup_{i=1}^r p_i(E). 
$$
But since $E$ is infinite, it is clear that $\cup_{i=1}^r p_i(E)$ 
is also infinite and thus Zariski dense in $B$. 
The point (i) is thus proved. 
\par
Now suppose that $R$ is analytically dense in a complex algebraic curve $C \subset \Tr_{K /\C}(A)$. 
The argument for the point (ii) is similar as above. 
The image $\pi(R) \subset B^{(r)}$ is now analytically dense 
in the image complex curve $\pi(C) \subset B^{(r)}$. 
The preimage $E = q_r^{-1}(\pi(R))$ is also 
analytically dense in the curve $q_r^{-1}(\pi(C))$. 
It follows that  $\cup_{i=1}^r p_i(E)$ is analytically dense in $B$ 
and the proof of (ii) and of Corollary \ref{t:raynaud-parshin-buium-strong-dense} 
is therefore completed. 
\end{proof}

\section{An effective criterion and a uniform finiteness in trivial families} 
\label{s:noguchi-uniform}

\subsection{Preliminaries}

The following auxiliary lemmata are well-known. 

\begin{lemma}
\label{l:quasi-finite-zariski} 
Let $f \colon X \to Y$ be a morphism of finite presentation of quasi-compact and separated schemes. 
Assume that $f$ is quasi-finite. 
 Then there exists a number $n \geq 0$ such that 
for every $y \in Y$, the fibre $X_y=f^{-1}(y)$ has at most $n$ points.  
\end{lemma}

\begin{proof}
By Zariski's Main Theorem (\cite[Chapter 4.4]{liu-alg-geom}),  
there exists a finite morphism $h \colon Z \to Y$ and 
an open immersion $\iota \colon X \to Z$ such that $f= h \circ \iota$. 
Hence, it suffices to prove the statement for $h$. 
We can clearly suppose   that $Z$ and $Y$ are integral. 
Then the degree $n_0=\deg h= [\kappa(Z) \colon \kappa(Y)] \in \N$ is well-defined.  
By generic flatness, there exists a dense open subset $U \subset Y$ 
such that $h_U \colon h^{-1}(U) \to U$ is flat and thus (cf. \cite[Exercise 5.1.25]{liu-alg-geom}) every fibre of 
$h_U$ has at most $n_0$ points. 
We continue the process with the finite morphism $h_1 \colon h^{-1}(Y_1) \to Y_1$. 
Similarly, we obtain another bound $n_1$ on the cardinality of the fibres of $h_1$ over 
a dense open subset $U_1$ of the proper closed subscheme 
$Y_1=Y \setminus U$ of $Y_0=Y$. 
Since $Y$ is a Noetherian topological space, the procedure terminates after finitely many steps (i.e., the closed subset  $Y_i \subset Y_{i-1}$ is empty for some $i \geq 1$). 
We can   
take $n$ to be the maximum of the $n_i$'s to conclude. 
\end{proof}

\begin{lemma}
\label{l:l:pairing-morphism}
Let $S$ be a Noetherian scheme. 
Let $X$ and $Y$ be respectively projective and quasi-projective 
schemes over $S$. 
Suppose that $X$ is flat over $S$ then the functor  
$\mathcal{M}or_S(X,Y)\colon$ $U \mapsto \Mor_U(X\times_S U, Y \times_S U)$ is representable by an open $S$-subscheme 
$\Mor_S(X,Y)$ of $\Hilb_{X \times Y/S}$. 
The following natural paring is an $S$-scheme morphism:  
\begin{align}
\label{l:natural-pairing-mor}
\sigma \colon X \times_S \Mor_S(X,Y) \to Y, \quad (x, f) \mapsto f(x). 
\end{align}
\end{lemma}

\begin{proof}
The first statement is a well-known  consequence of Grothendieck's theorem on Quot-schemes 
(cf. \cite[Les sch\' emas de Hilbert]{fga-4-hilbert}). 
 The second statement is an application of Yoneda's lemma. 
Let $Z=X \times_S \Mor_S(X,Y)$ and consider 
the covariant transformation of point functors  $\Sigma \colon h_{Z} \to h_Y$ given by composition:  
\begin{align*}
\Sigma(U) \colon h_Z(U)=X(U) \times \Mor_S(X,Y)(U) &\longrightarrow h_Y(U)=Y(U)\\
(x, f) & \longmapsto f \circ x,
\end{align*}
for every $S$-scheme $U$. 
Here we identify naturally $X(U)\simeq X_U(U)$ and $Y(U) \simeq Y_U(U)$ by the universal property of 
fibre products. 
Observe that $\Sigma$ is exactly the pairing \eqref{l:natural-pairing-mor}.  
It is a direct verification that $\Sigma$ is a natural transformation, i.e., for every $S$-scheme morphism  
$q \colon U \to V$, we have $\Sigma(V) \circ h_Z(q)= h_Y(q) \circ \Sigma(U)$. 
Therefore, Yoneda's lemma implies that $\Sigma$ is representable by an $S$-scheme morphism 
$\sigma \colon Z \to Y$.  
\end{proof}

\subsection{Two uniform lemmata on abelian varieties}

A \emph{polarization} on an abelian variety is the algebraic equivalence class of an ample line bundle, or equivalently, an ample line bundle defined up to translation. 
Let $L$ be an ample invertible sheaf on an abelian variety $A$ of dimension $n$. 
Associated to $L$ there is a canonical homomorphism 
$\phi_L \colon A \to A^\vee$ defined by $a \mapsto \tau_a^*L \otimes L$. 
We have by the Riemann-Roch theorem that $h^0(A,L)=(L^n)/n!$, and $h^i(A,L)=0$ for all $i >0$ 
(cf. the canonical \cite{fulton-98} for the intersection theory). 
Moreover, $\deg \phi_L = (L^n/n!)^2$ and we call this number the \emph{degree} of the polarization induces by $L$.  
%In particular, $n!$ divides $(L^n)$ and $L$ admits global sections.   
%$L$ is a \emph{principal polarization} if $(L^n)=n!$, i.e., $h^0(A,L)=1$, or equivalently, $\phi_L$ is an isomorphism. 
\par
We   recall the following well-known result of Mumford on the moduli space of abelian varieties.  
Let $n,d,N \geq 1$ be integers. 
For each $\Z[1/N]$-scheme $T$, 
let $\mathcal{M}_{g,d^2,N}(T)$ 
be the set of isomorphism classes of triples $(A,\phi, j)$, where $A$ 
is an abelian scheme over $T$ of relative dimension $n$, 
$\phi$ is a degree $d^2$ polarization, 
and $j \colon (\Z/N\Z)^{2n} \to A[N]$ is an isomorphism of $T$-groups. 
We thus obtain a contravariant functor $\mathcal{M}_{g,d^2,N} \colon \Sch_{\Z[1/N]} \to \Set$.  
\begin{theorem}
[Mumford]
\label{t:mumford-moduli-abelian}
The functor $\mathcal{M}_{n,d^2,N}$ is representable by a quasi-projective $\Z$-scheme $M_{n,d^2,N}$ 
whenever $N \geq 6^nd\sqrt{n!}$. 
\end{theorem}

\begin{proof}
See \cite[Theorem 7.9]{mumford-GIT} 
\end{proof}

Now return to the notations of Theorem \ref{t:uniform-abelian-trivial}. 
Denote $\mathcal{A}=A \times B$ the constant abelian family over $B$ induced by $A$. 
Let $K= k(B)$ and $A_K= A \otimes K$ the generic fibre of $\mathcal{A}$. 
We can thus view $W$ as a subset of rational points of $A(K)$. 
The strategy is to establish first a uniform bound on the canonical height associated to the symmetric ample divisor $D + [-1]^*D$. 
 We shall need the following uniform lemma for the proof of Theorem  \ref{t:uniform-abelian-trivial}. 

\begin{lemma}
\label{l:uniform-dominate-symmetric-ample}
For each integers $n, d \geq 1$, there exists a number $m(n, d)$ with the following property. 
For every ample divisor $D$ of degree $d$ on an abelian variety $A$ of dimension $n$, 
 $$
mD - (D + [-1]^*D)
$$ 
is ample whenever $m \geq m(n,d)$.  
\end{lemma}

Notice that if $n=1$, i.e., $A$ is an elliptic curve, then 
we can simply take $m(n,d)=3$. Indeed, since $\deg D = \deg [-1]^*D$, 
$$\deg ( 3D-(D + [-1]^*D) )= \deg (2D -  [-1]^*D)= \deg D >0. $$
It follows that $3D-(D + [-1]^*D)$ is ample.

\begin{proof} [Proof of Lemma \ref{l:uniform-dominate-symmetric-ample}]
 We consider a 
parameter space of $n$-dimensional abelian varieties 
$\pi \colon \mathcal{A} \to T$ which is also an \emph{abelian scheme} over $T$, i.e., 
it is equipped with a zero section $e \colon T \to \mathcal{A}$, 
an inverse morphism $[-1] \colon \mathcal{A} \times_T \mathcal{A} \to \mathcal{A}$ 
and an associative abelian sum morphism $\sigma \colon \mathcal{A} \times_T \mathcal{A} \to \mathcal{A}$. 
Let $\DD$ be the universal effective ample divisor of degree $d$ on $\mathcal{A}$. 
For example, we can take $\mathcal{A}$, $\DD$ to be respectively the universal abelian scheme 
with large enough level structure $N$ 
and the universal ample divisor of degree $d^2$ over the Mumford moduli space $T=M_{n, d^2, N}$ 
in Theorem \ref{t:mumford-moduli-abelian}. 
\par
It suffices to find $m=m(n,d)$ such that $m\DD - (\DD + [-1]^*\DD)$ is ample. 
This can be done by Noetherian induction as follows. 
Up to replacing $T$ by each of its irreducible components, 
we can suppose that $T$ is integral.  
Let $\eta$ be the generic point of $T$. 
Then $\DD_\eta$ is ample on the generic abelian variety $\mathcal{A}_\eta$. 
Hence, there exists an integer $m_0 \geq 1$ such that 
$m_0\DD_{\eta} - [-1]^*\DD_{\eta}$ is  
ample. 
As ampleness is an open property on the base,  
there exists a Zariski dense open subset $U_1 \subset T$ 
over which $m_0\DD- [-1]^*\DD$ is ample. 
Let $T_1= T \setminus U_1$. 
As $T_1$ contains only a finite number of irreducible components, we can suppose that 
$T_1$ is irreducible to simplify the notations without affecting the uniform conclusion on $m(n,d)$. 
We continue the above procedure for  
 $T_1$ and for the abelian scheme $\pi_1  \colon \mathcal{A}_1 \to T_1$ and for 
and $\DD_1= \DD\vert_{T_1}$, where $\pi_1= \pi \vert_{\pi^{-1}T_1}$ and $ \mathcal{A}_1= \pi^{-1}T_1$.  
We then obtain a closed subset 
$T_2 \subset T_1$ (which we can again assume irreducible without loss of generality) and 
a constant $m_1$ and so on. 
In this ways, we obtain 
a sequence of integers $m_i \geq 1$ and a sequence of 
decreasing closed subsets $T_0=T \supset T_1 \supset \cdots$ such that the divisor 
$$\left(\max_{0 \leq i \leq p}m_i \right) \DD- [-1]^*\DD$$ 
is relatively ample over $T\setminus T_{p+1}$. 
Notice that $\dim T_{i+1} < \dim T_i$ for every $i$. 
It follows that the sequence $(T_i)_i$ is finite, i.e., $T_{q+1}=\varnothing$ for some integer $q \geq 0$. 
It suffices to set $m=m(n,d)=1+ \max_{0 \leq i \leq q} m_i$ to see that 
$$m\DD - (\DD + [-1]^*\DD)= \left(\max_{0 \leq i \leq q}m_i \right) \DD- [-1]^*\DD$$ 
is relatively ample over $T\setminus T_{q+1}=T$. 
Therefore, $m\DD_t - (\DD_t + [-1]^*\DD_t)$ is ample for every $t \in T$ and the conclusion follows. 
\end{proof}

Now let $A$ be a complex abelian 
variety of dimension $n$. 
Let $L$ be an effective ample \emph{numerical class} of divisors on $A$ of degree $d=L^n/n!$. 
For each finite subset $S \subset B$ and every effective divisor $D \in L$, 
let $Z(S,D)$ be the set of nonconstant morphisms $f \colon B \to A$ such that 
$f^{-1}(D) \subset S$. In the theorem below, the notation $D \geq 0$ 
means that $D$ is an effective divisor. 
As an application, we   show that: 

\begin{lemma}
\label{p:abelian-trivial}
 For every $s \in \N$,  we have 
$$ \# \cup_{0 \leq D \in L} \cup_{\#S \leq s\,} Z(S,D) \text{ modulo } A(\C) \leq (2\sqrt{m(n,d)sM_{NW}(g,n,d)} +1)^{4ng}$$
where $M_{NW}(g,s,n,d)$ and $m(n,d)$ are defined in Theorem \ref{t:noguchi04} and Lemma \ref{l:uniform-dominate-symmetric-ample} respectively.  
\end{lemma}

\begin{proof}
For any ample divisor $D\in L$, we known that $m(n,d) D - (D + [-1]^*D)$ is ample by Lemma  
 \ref{l:uniform-dominate-symmetric-ample}. 
It follows from Theorem \ref{t:noguchi04} that the degree with respect to $D$ of any nonconstant 
$(S,D)$-integral point $P \in A(K)$ is bounded by $sM_{NW(g,n,d)}$. 
Therefore, the canonical height of such points with respect to the symmetric ample divisor 
$D + [-1]^*D$ is bounded uniformly above by $m(n,d)sM_{NW}(g,n,d)$. 
We can thus conclude by a standard argument using the Lenstra Counting Lemma \cite[Lemma 6]{silverman-lenstra-counting-lemma}    
applied to the lattice 
$A(K)/A(\C)$ and the canonical positive definite quadratic form associated to $D + [-1]^*D$ 
(cf. \cite{hindry-silverman-book} and \cite{hindry-silverman-88}). 
\end{proof}

\subsection{Proof of Theorem \ref{t:uniform-abelian-trivial}} 

 We are now in position to prove the fourth main result of the article.  

\begin{proof}[Proof of Theorem \ref{t:uniform-abelian-trivial}] 
Let $ f \colon (B \setminus S) \to (A \setminus D)$ be a nonconstant morphism. 
By Theorem \ref{t:noguchi04}, we have $\deg_B f^*D \leq H \coloneqq s M_{NW}(g,n,d)$. 
Since $D$ is ample, it follows that 
\begin{align}
\label{e:condition-fiber}
f \in  \coprod_{1 \leq N \leq H} \Mor^{Nt+1-g}(B, A),
\end{align}
where the Hilbert polynomial is calculated with respect to the line bundle $\OO(D)$. 
\par
By Theorem \ref{t:mumford-moduli-abelian}, there exists 
a universal abelian variety $A(n,d^2,N) \to M_{n,d^2,N}$ of $n$-dimensional degree-$d^2$ polarized abelian varieties of large level $N$-structure. 
Let $\mathcal{L}$  be the universal degree-$d^2$ ample line bundle on $A(n,d^2,N)$.  
Then $\DD \to Y$, where $Y \coloneqq A(n,d^2,N)  \times (\mathrm{Proj} \mathcal{L})^*$, 
is the universal effective ample divisor of degree $d^2$ on $A(n,d^2,N)$. 
In what follows, we denote  
$\mathcal{A}=A(n,d^2,N)\times Y $ and $T_1=  M_{n,d^2,N} \times Y $. 
Let $T_2$ be a coarse moduli space of curves of genus $g$ with the universal curve $\CC \to T_2$. 
\par
Let $T= T_1 \times T_2$ and let $p_1$, $p_2$ be respectively the first and second projection. 
Up to making the base change $p_1 \colon T \to T_1$ for $\mathcal{A}, \DD, T_1$ and 
$p_2 \colon T \to T_2$ for $\CC, \sigma, T_2$, we can suppose that $T_1=T_2=T$. 
Define 
$$\mathcal{M}= \coprod_{1 \leq N \leq H} \Mor_T^{Nt+1-g}(\CC, \mathcal{A}),$$
then $\mathcal{M}$ is a quasi-projective $T$-scheme by the standard Hilbert scheme theory.  
The Hilbert polynomial is calculated using the line bundle $\OO(\DD)$. 
Now for each $i=1, \dots, s$, consider   the image $\Delta_i$ of the following 
closed immersion of $T$-schemes (diagonal morphisms)  
\begin{align*}
 h_i \colon \CC \times_T (\CC \times_T \cdots \times_T \hat{\CC_i} \times_T \cdots \times_T \CC )  
 &\longrightarrow  \CC \times_T \CC_T^{s}\\
 (x, (x_1, \dots, x_{i-1}, x_{i+1}, \dots, x_s) ) & \longmapsto (x, (x_1, \dots, x_{i-1}, x , x_{i+1}, \dots, x_s)), 
 \end{align*}
 where $\hat{\CC_i}$ denotes the omitted $i$-th copy of $\CC$ in the fibre product $\CC_T^s$. 
Let $\Delta= \cup_{1 \leq i \leq s} \Delta_i$ be a closed subset of $\CC \times_T \CC_T^{s}$.  
Consider the following 
map 
\begin{align*}
\Psi \colon \DD \times_T (\CC \times_T  \CC_T^s \setminus \Delta) \times_T \mathcal{M} 
& \longrightarrow \mathcal{A}\times_T  \CC_T^s \times_T \mathcal{M}\\
(x, y, z, f ) & \longmapsto (x-f(y), z, f). 
\end{align*}

We claim that $\Psi$ is a morphism of quasi-projective $T$-schemes. 
 Indeed, $\Psi$ is a composition of the following $T$-morphisms: 
$$
\iota \colon \DD \times_T (\CC \times_T  \CC_T^s \setminus \Delta) \times_T \mathcal{M} 
\to  \DD \times_T \CC \times_T  \CC_T^s \times_T \mathcal{M} 
\simeq  
\DD \times_T \CC  \times_T \mathcal{M} \times_T  \CC_T^s \quad \text{(immersion)},
$$ 
$$ 
\Id \times ( \sigma, \pr_2) \times \Id  \colon \DD \times_T \CC \times_T \mathcal{M}  \times \CC_T^s
\to \DD \times_T \mathcal{A}\times_T \mathcal{M} \times \CC_T^s,
 \quad (\text{cf. Lemma } \ref{l:l:pairing-morphism})
$$
and the subtraction morphism $\mathcal{A} \times_T \mathcal{A} \to \mathcal{A}$ 
 restricted to the first two factors $\DD \times_T \mathcal{A}$.  
\par
Now define $\Sigma \coloneqq \mathcal{A}\times_T  \CC_T^s \times_T \mathcal{M} \setminus \im \Psi$ to be the 
complement of the image of $\Psi$. 
By the Chevalley theorem,  $\im \Psi$ is a quasi-projective $T$-scheme thus so is $\Sigma$.  
Let $\pi \colon \Sigma \to  \CC_T^s \times_T \mathcal{M} \times_T T$ be the induced projection morphism. 
By \cite[Proposition 9.2.6.(iv)]{ega-4-3}, the subset $E \subset \CC_T^s \times_T \mathcal{M} \times_T T$ above 
which $\pi$ has finite fibres is constructible. 
We equip $E$ with the induced reduced scheme structure. 
Then (cf. Lemma \ref{l:quasi-finite-zariski} below) there exists a number $N_1(g,s,n,d)$ such that for ever $q \in E$, we have 
\begin{equation}
\label{e:N-1-main-constant-noguchi}
\#\Sigma_q=\# \pi^{-1}(q) \leq N_1(g,s,n,d). 
\end{equation}

Lemma \ref{p:abelian-trivial} implies that for each data set $(B, S, A, D)$ with prescribed invariants 
$(g,s,n,d)$, the number of translation classes of nonconstant morphisms  
$f \colon (B \setminus S) \to (A \setminus D)$ is bounded above by $(2\sqrt{m(n,d)sM_{NW}(g,n,d)} +1)^{4ng}$ 
where $M_{NW}(g,s,n,d)$ and $m(n,d)$ are defined respectively in 
Theorem \ref{t:noguchi04} and Lemma \ref{l:uniform-dominate-symmetric-ample}.  
\par
Moreover, from 
 \eqref{e:condition-fiber} and the definitions of $T_1$ and $T_2$, Theorem \ref{t:noguchi04} implies that there exists $t \in T$ such that some translate of such $f$ belongs to $\mathcal{M}_t$. 
On the other hand, 
the constructions of $\Psi$, $\Sigma$ and of the set $E$ which satisfies \eqref{e:N-1-main-constant-noguchi} imply  
that either  
there are at most $N_1(g,s,n,d)$ translates of $f$ that belong to $W$ or either 
there are an infinite number of them. 
By Theorem \ref{c:raynaud-parshin-buium-strong}, the latter case is excluded if no translates of $\im f$ are contained in $D$ . 
This allows us to conclude the first two statements (i) and (ii) by setting 
$$N(g,s,n,d)=N_1(g,s,n,d) (2\sqrt{m(n,d)sM_{NW}(g,n,d)} +1)^{4ng}.$$ 

For the last statement (iii), suppose that $n=2$, $d> 2g-2$ and that $D$ is an integral curve. 
If $g=0$ then $B \simeq \Proj^1$. 
Since abelian varieties do not contain 
rational curves, 
every morphism $\Proj^1 \to A$ is constant and thus 
$\#W=0$.  
Suppose now that $g \geq 1$. 
Since the canonical divisor $K_A$ of $A$ is trivial, we deduce from the 
adjunction formula that 
$$p_a(D)=\frac{D^2+D\cdot K_A}{2}+1=\frac{d}{2}+1 > g \geq 1.$$
Therefore, the Riemann-Roch theorem for curves implies that any algebraic morphism 
$f \colon B \to D$ must be constant. 
It follows that for any non constant morphism $f \in W$, no translates of $\im f$ can be contained in $D$.  
Hence, $W$ must be finite by (ii)  and we can conclude from (i)  that  
$\# W \leq N(g,s,n,d)$. 
\end{proof}

\section{A semi-effective bound on the number of integral points}

\subsection{Jet differentials} 
 
Let $V$ be a complex algebraic variety. 
Following Noguchi's notations, for each integer $k \geq 1$, 
we regard the $k$-jets over a point $x \in V$ as morphisms from 
$\Spec \C\{t\}/(t^{k+1})$ to $V$ sending the geometric
point of $\Spec \C\{t\}/(t^{k+1})$ to $x$. 
The $k$-jet space is defined as $J^k(V)= \Mor(\Spec \C\{t\}/(t^{k+1}), V)$. 
We denote by $\frak{m}(k)$ the maximal ideal $(t)$ of $\C\{t\}/(t^{k+1})$. 
When $V$ is a complex manifold, the space of $k$-jets is defined 
similarly using local charts and holomorphic germs of maps $\C \to V$.  
\par
If $Y$ is a closed complex subspace of a complex manifold $X$, then $J^k(Y)$ 
is a naturally defined closed complex subspace of $J^k(X)$ for every $k \in \N$. 
See \cite{noguchi-winkelmann-04} for more details.  

\begin{lemma}
\label{l:algebraic-jet-subvariety}
Let $Y$ be a complex closed subspace of a complex space $X$. 
Assume that $f \colon X \to Z$ is a biholomorphism with $Z$ a complex quasi-projective algebraic variety. Then $f(Y)$ is an algebraic closed subset of $Z$. 
\end{lemma}

\begin{proof}
We regard $Z$ as a complex space then since $f$ is a biholomorphism, 
$f(Y)$ is also a complex closed subspace of $Z$. 
As $Z$ is quasi-projective, there exists an algebraic   immersion $\phi \colon Z \to \Proj^n$ for some integer $n \geq 1$. 
Let $T $ be the analytic closure of $f(Y)$ in $\Proj^n(\C)$. 
Since $f(Y)$ is analytically closed in $Z$,  $T \cap Z= f(Y)$. 
By Chow's theorem (cf. \cite[Theorem 1.1.3]{neeman}), $T$ is  an algebraic closed subvariety 
of  $\Proj^n(\C)$. 
It follows that $f(Y)= T \cap Z$ is also an algebraic closed subvariety of $Z$.  
\end{proof}

\subsection{Intersection multiplicities with general divisors in constant families}

We mention here without proof the following well-known property. 

\begin{lemma}
[Identity principle] 
\label{l:identity-principle}
Let $f \colon B \to X$ be a holomorphic map from $B$ to a 
complex manifold $X$. 
Let $b \in B$ and let $Y$ be an effective divisor of $X$ such that $\text{mult}_{b} f^* Y \geq n$ for all $n\in \N$.  Then $f(B) \subset Y$.  
\end{lemma}

Adopting the approach of Noguchi-Winkelmann in the proof Theorem \ref{t:noguchi04}, 
we obtain the following uniform result (Theorem \ref{t:NW-generalized-intro}) 
on local intersection 
multiplicities in the setting of a general divisor in a  constant family of abelian varieties. 

\begin{theorem}
\label{t:NW-generalized}
Let $A$ be a complex abelian variety. 
Let $\DD$ be an integral divisor in $A \times B$. 
There exists a number $M >0$ satisfying the following property.  
For every 
morphism $\phi \colon B \to A$ such that $(\phi \times \Id_B)(B) \nsubseteq \DD$, 
we have 
$$ \mult_b (\phi \times \Id)^* \DD \leq M \quad \text{ for all } b \in B.$$  
\end{theorem}

\begin{proof} 
Let $\pi \colon A\times B \to B$ be the canonical projection. 
Let $G = J(B)$ be the Jacobian variety of $B$ and fix an embedding $B \to G$. 
For complex spaces $X, Y$ and for every integer $n \geq 1$, we have a canonical identification 
$J^n(X \times Y)= J^n(X) \times J^n(Y)$. 
Let $\mathfrak{m}(n)$ be the maximal ideal of $\C\{t\}/(t^{n+1})$. 
Then by the exponential map, we have a natural 
\emph{holomorphic} trivialization of the $n$-th jet bundle of $A$ 
(cf.  \cite[Proposition 3.11]{noguchi-winkelmann-04}): 
$$J^n(A)\simeq A \times ( \mathfrak{m}(n) \otimes \Lie A).$$
Similarly, we have $J^n(G) \simeq G \times ( \mathfrak{m}(n) \otimes \Lie G )$. 
This makes $J^n(A)$ and $J^n(G)$ become algebraic subvarieties 
since for example $ \mathfrak{m}(n) \otimes \Lie A \simeq \C^{n \dim A}$ is algebraic.  
\par
Let $\phi \colon B \to A$ be an algebraic morphism. 
Then $\phi$ induces canonically a homomorphism 
$\tilde{\phi} \colon G \to A$ and a morphism $d^n \tilde{\phi} \colon J^n(G) \to J^n(A)$.  
Moreover, since morphisms from $G$ to $A$ \emph{lift} canonically to 
homomorphisms of the associated Lie algebras, there exists 
 {a canonical linear map} $\varphi_{\phi} \colon \Lie G \to \Lie A$ (as $\Lie A $, $\Lie G$ are commutative) which is 
independent of $n$ and 
such that (cf.  \cite[Lemma 4.1]{noguchi-winkelmann-04}): 
$$
\Phi=\tilde{\phi} \times (\Id_{\mathfrak{m}(n)} \otimes \varphi_{\phi}) \colon G \times ( \mathfrak{m}(n) \otimes \Lie G ) \to A \times (\mathfrak{m}(n) \otimes \Lie A)
$$
is compatible with $d^n \tilde{\phi}$, i.e., such that the following diagram is commutative:  
 \[ 
\label{d-commute-diagram-algebrization}
\begin{tikzcd}
J^n(G)  \arrow[r, , "d^n \tilde{\phi}"] \arrow[d, "\simeq"] 
 &  J^n(A)    \arrow[d, "\simeq"] \\
  B \times ( \mathfrak{m}(n) \otimes \Lie G ) \arrow[r, "\Phi"]  
  & A \times (\mathfrak{m}(n) \otimes \Lie A). 
\end{tikzcd}
\]
Notice that the linear map $\varphi_{\phi}$ determines  the morphism $\phi$ up to a translation of $A$. 
By abuse of notations, we   denote simply $d \phi \colon \Lie G \to \Lie A$ the linear map $\varphi_\phi$ for 
every morphism $\phi \colon B \to A$. 
 \par
Denote $V= \Hom_{lin}(\Lie G , \Lie A) \simeq \Hom_{\C}(\C^g, \C^{\dim A})\simeq \A^{g\dim A}$. 
Consider   the algebraic variety $T \coloneqq B \times_B \DD \times V$ where the morphism 
$\DD \to B$ is induced by the projection $\pi \colon A \times B \to B$. 
We have natural inclusions of closed complex spaces: 
$$
J^n(\DD) \subset J^n(A \times B) \subset J^n(A \times G)= J^n(A) \times J^n(G), 
\quad \quad J^n(B) \subset J^n(G). 
$$
Using the trivializations of $J^n(A)$, $J^n(G)$, 
 the sets $J^n(\DD)$ and $J^n(B)$ can be regarded respectively as {closed algebraic subvarieties} 
of $\DD \times \left( \mathfrak{m}(n) \otimes (\Lie A \times  \Lie G) \right)$ and $B \times \left( \mathfrak{m}(n) \otimes \Lie G\right) $ 
(cf. Lemma \ref{l:algebraic-jet-subvariety}).  
 For each integer $n \geq 1$,   define 
\begin{equation*}
W_n \coloneqq \{ (b, d, \varphi) \in T \colon 
\left( (\Id_{\mathfrak{m}(n)} \otimes  \varphi)\times \Id_{J^n(G)} \right) (J^n_b(B)) 
\subset J_d^{n} (\DD)\} \subset T.
\end{equation*}

We claim that $W_n$ is a closed algebraic subset of $T$ for every $n \geq 1$. 
Indeed, we have the valuation morphism $V \times \Lie G  \to \Lie A$, $(\varphi, v) \mapsto \varphi(v)$ which is clearly algebraic. 
Consider the following $T$-morphism of algebraic varieties
\begin{align*}
\label{e:algebraization-noguchi-vary-d}
\Psi \colon 
J^n(B) \times_B T &  \to 
 \left( \mathfrak{m}(n) \otimes (\Lie A \times  \Lie G) \right) \times T
 \\
\left( \sum_{i} \alpha_i \otimes v_i, b, d, \varphi \right) & \mapsto \left(\sum_i \alpha_i \otimes  (\varphi(v_i), v_i), b, d,\varphi \right) .
\end{align*} 
Denote $\Sigma \coloneqq \im \Psi \setminus J^n(\DD)\times_\DD T$. 
Let $p_T \colon  \left( \mathfrak{m}(n) \otimes (\Lie A \times  \Lie G) \right) \times T \to T$ be the second projection then      
$W_n= T \setminus p_T(\Sigma)$. 
By the Chevalley theorem applied to $\Psi$ and $p_T$, $W_n$ is an algebraic subset of $T$. 
We shall show that the induced map $p \colon \im \Psi \to T$ is flat. 
This will prove the claim since $p_T(\Sigma)$ is then a Zariski open subset of $T$ by the open property of flat morphisms and as $\Sigma$ is Zariski open in $\im \Psi$. 
\par
Clearly, it suffices to prove that the image of the map 
\begin{align*}
 \Omega \colon 
J^n(B) \times_B B \times V &  \to 
 \left( \mathfrak{m}(n) \otimes (\Lie A \times  \Lie G) \right) \times B \times V \\
\left( \sum_{i} \alpha_i \otimes v_i, b,\varphi \right) & \mapsto \left(\sum_i \alpha_i \otimes  (\varphi(v_i), v_i), b, \varphi \right) 
\end{align*} 
is smooth (and thus flat) over $B \times V$. 
But the fibre of $\im \Omega$ over  $(b, \varphi) \in  B \times V$ is smooth as the image under the linear map 
$\Id_{\mathfrak{m}(n)} \otimes (\varphi, \Id_{\Lie G})$ of the smooth linear subspace $J^n_b(B) \subset \mathfrak{m}(n)\otimes \Lie G$. 
Hence,     $W_n$ is indeed a closed algebraic subset of $T$  as claimed. 
\par
Now we verify that $(W_n)_{n \geq 1}$ is a descending sequence of algebraic varieties. 
Indeed, suppose that $(b,d, \varphi) \in W_{n+1}$, this means by the definition that 
$$ d^{n+1} (f \times \Id_B) (J_b^{n+1}(B)) \subset J_d^{n+1} (\DD)$$
for any holomorphic germ map $f \colon (B,b) \to (A,d)$ such that  
$d f_b (v) = \varphi(v)$ for all $v \in \Lie G$. 
Clearly, this is just a reformulation of the local condition:  
$$\mult_b(f \times \Id_B)^* \DD \geq n+1.$$ 
It follows that $(b,d,\varphi) \in W_n$ and thus $W_{n+1} \subset W_n$ as desired. 
\par
Therefore,  
since the Zariski topology is Noetherian, the sequence $(W_n)_{n \geq 1}$ is stationary for $n \geq M$ for some 
integer $M \geq 1$. 
If a morphism $f \colon B \to A$ verifies $\mult_b(f \times \Id_B)^* \DD \geq M+1$ for some $b \in B$ then $(b, (b,f(b)), df_b) \in W_n$ for all $n \geq M$. 
This implies that $\mult_b(f \times \Id_B)^* \DD \geq n$ for all $n \geq M$ and we must have $(f \times \Id_B)(B) \subset \DD$ by Lemma \ref{l:identity-principle}. 
The conclusion thus  follows. 
\end{proof}

\subsection{Semi-effective bound for integral points with varying divisors} 

We remark the following elementary standard lemma that we omit the proof. 

\begin{lemma}
\label{l:graph-intersection}
Let $X$ be a  complex proper algebraic variety and let $D$ be an effective integral divisor on $X$. 
For every morphism $f \colon B \to X$, we have: 
$$ \deg_B f^*\OO(D)= (f_*B) \cdot D = ((f\times \Id_B)_* B ) \cdot (D \times B). $$
\end{lemma}

We continue a comparison result for the positivity of divisors. 

\begin{lemma}
\label{l:linearize-trivial-abelian}
 Let $A$ be a complex abelian variety. 
Let $\DD \subset A \times B$ be an effective integral divisor such that 
$\DD_K$ is ample.  
Let $b_0 \in B$. 
Then $D= \DD_{b_0}$ is an ample divisor on $A$ and 
there exists $N, c >0$ such that   
 for every morphism $f \colon B \to A$ and $\tilde{f} = f \times \Id_B$, we have: 
\begin{equation}
\label{e:lemma-larzarsfeld-relative-ample}
-c + N^{-1} (\tilde{f}_*B) \cdot \DD  \leq  (f_*B) \cdot D   \leq  N  (\tilde{f}_*B) \cdot \DD +c.
\end{equation} 
\end{lemma}

\begin{proof}
Since $\DD$ is integral and $B$ is a Dedekind scheme, 
$\DD \subset A \times B$ is a flat family of effective Cartier divisors parametrized 
by $B$ (cf.  \cite[Proposition 3.9]{liu-alg-geom}). 
For $b \in B$, the divisors $\DD_b$ are all algebraically equivalent on $A$. 
Moreover, they are ample 
since $\DD_K$ is ample.  
Hence, for $N = 2$, the divisors $N\DD_b -D$ and $N D - \DD_b$ are 
ample on $A_b=A$ for every $b \in B$.  
It follows that $N \DD - D \times B$ and $N D \times B - \DD$ are relatively  
ample over $B$. 
Therefore, there exists vertical divisors $V_1$ and $V_2$ in $A \times B$ 
with finite images on $B$ such that 
$\OO( N\DD - D \times  B + V_1) $ and 
$\OO( N D \times B -\DD +V_2) $ are ample 
(cf. \cite[Theorem 1.7.8, Proposition 1.7.10]{lazarsfeld-positivity-I}).  
Let $c= \max( |\deg_B \pi_* V_1|, |\deg_B \pi_*V_2|) \in \N$. 
For every morphism $f \colon B \to A$ and $\tilde{f} = f \times \Id_B$, we deduce that:  
\begin{align}
\label{e:compare-horizontal-vertical}
0  \leq (\tilde{f}_*B)  \cdot (N\DD- D \times B + V_1)  \leq N (\tilde{f}_*B) \cdot \DD  - N (\tilde{f}_*B) \cdot (D \times B) + c.   
\end{align}
The second inequality follows from the fact that $\tilde{f}_*B$ 
is reduced and transverse to every vertical  integral divisor $V$ relative to $A\times B \to B$. 
Since 
$(\tilde{f}_*B) \cdot (D \times B)= f_*(B) \cdot D$ 
 by Lemma \ref{l:graph-intersection}, we deduce from \eqref{e:compare-horizontal-vertical} that 
$$  (f_*B) \cdot D   \leq  N (\tilde{f}_*B) \cdot \DD +c.$$
The left inequality of \eqref{e:lemma-larzarsfeld-relative-ample} is obtained similarly 
be considering the ample line bundle $\OO( N D \times B -\DD +V_2) $.  
\end{proof}

\subsection{Proof of Corollary \ref{p:abelian-nontrivial-intro}}
 \label{s:semi-bound-integral-point}
 
Using Lemma \ref{l:linearize-trivial-abelian} and Theorem \ref{t:NW-generalized}, we can now give the proof  
of the semi-effective bound on integral points in contant abelian varieties 
with respect to a general effective divisor in Corollary \ref{p:abelian-nontrivial-intro}. 
 \par
Let $A$ be an abelian 
variety of dimension $n$. 
Let $\mathcal{A} = A \times B$ and let $\DD \subset \mathcal{A}$ 
be an effective divisor such that $\DD_K$ is ample. 
For each morphism $f \colon B \to A$, we denote 
$\tilde{f} = f\times \Id_B \colon B \to \mathcal{A}$. 
For each integer $s \geq 1$, 
let $W(s, \DD)$ be the set of morphisms $f \colon B \to A$ such that 
$\# \tilde{f}(B) \cap \DD \leq s$.

\begin{proof}[Proof of Corollary \ref{p:abelian-nontrivial-intro}]
Since we work modulo translations by $A$, we can always and do assume   
$\tilde{f}(B) \nsubseteq \DD$ for each morphism $f \colon B \to A$ after a suitable translation. 
Indeed, suppose on the contrary that $\tilde{f_a}(B) \subset \DD$ for all $a \in A$ where 
$f_a $ denotes 
the composition of $f$ with the translation-by-$a$ map $t_a \colon A \to A$, $x \mapsto x+a$. 
Then   
$ \mathcal{A} \subset \cup_{a \in A} \tilde{f_a}(B) \subset  \DD$, 
which is a contradiction since $\dim \DD= \dim \mathcal{A} -1$. 
\par
Theorem \ref{t:NW-generalized}   implies that there exists a constant $M >0$ such that 
$ \mult_b \tilde{f}^* \DD \leq M$ for every $ b \in B$ 
and for every $f \in W(s, \DD)$ with $ \tilde{f}(B) \nsubseteq \DD$. It follows that: 
\begin{align}
 \label{e:recall-comparison-to-constant}
\deg_B \tilde{f}^* \DD = \sum_{b\in B} \mult_b \tilde{f}^*\DD 
\leq  \left(\# \tilde{f}(B) \cap \DD \right).M 
\leq sM. 
\end{align}
 
By Lemma \ref{l:linearize-trivial-abelian}, there exists $N,c >0$ such that 
$ \deg_B f^* \OO( D )  \leq N \deg \tilde{f}^*\DD +c$  
where $D=\DD_{b_0} \subset A_{b_0}=A$ for some $b_0 \in B$. 
By the same lemma \emph{loc. cit}, $D$ is ample.   
Combining with \eqref{e:recall-comparison-to-constant}, 
we find that for every $f \in W(s, \DD)$ such that $ \tilde{f}(B) \nsubseteq \DD$:   
\begin{equation}
\label{e-symmetric-ample-on-constant-abelian}
\deg_B f^* \OO( D )   \leq sNM+c \leq s(NM+c).
\end{equation}
Since $D$ is ample, $mD$ dominates a symmetric ample divisor $D_0$ on $A$ for some integer $m \geq 1$. 
We identify each $f \in W(s, \DD)$ with the associated rational point $P_f \in A(K)$. 
Let $H=m(NM+c)$. Then the canonical height of $P_f$ in $A(K)$ with respect to $D_0$, which 
is given by $\deg_B f^*\OO(D_0)$ and is invariant to the translation class of $f$,    
is bounded from the above  by $sH$ by \eqref{e-symmetric-ample-on-constant-abelian}. 
A standard  counting argument on the lattice $A(K)/ A(\C)$ 
(remark that all torsion points of $A(K)$ belong to $A(\C)$) using the 
Lenstra Counting Lemma \cite[Lemma 6]{silverman-lenstra-counting-lemma}  
completes the proof (cf. \cite{hindry-silverman-book} and \cite{hindry-silverman-88}). 
\end{proof}

\section{An application on the generic emptyness of integral sections} 
 \label{s:generic-empty-S-abelian-trivial}

Let $A$ be a complex abelian 
variety.  
Let $\mathcal{A} = A \times B$ and let $\DD \subset \mathcal{A}$ 
be an effective divisor such that $\DD_K$ is ample. 
The following result says that the set of $(S, \DD)$-integral points on $\mathcal{A}$ 
is empty for a general choice of the finite subset $S \subset B$. 
Moreover, there are much less integral points as the size $s= \#S$ of $S$ goes to infinity. 
More precisely, the result says that for each $s \in \N$,  
the space of subsets $S \subset B$ of cardinality $s$ 
such that there exists $(S, \DD)$-integral points 
lies in a closed subspace of dimension at most $\dim A$ in the 
$s$-dimensional parameter space $B^{(s)}$.

\begin{corollary}
\label{t:generic-empty-S-abelian-trivial}
Assume   $\DD_K(K)= \varnothing$. 
For every $s \in \N$, there exists a closed algebraic subset $V_s \subset B^{(s)}$ 
such that the following hold: 
\begin{enumerate} [\rm (i)] 
\item
$\dim V_s \leq \dim A$; 
\item 
for $S\in B$ a finite subset of cardinality $s$ and $U= B \setminus S$,  
if the projection $(\mathcal{A} \setminus \DD)\vert_U \to U$  
admits an algebraic section then 
the image of $S$ in $B^{(s)}$ belongs to $V_s$. 
\end{enumerate}
\end{corollary}

\begin{proof}[Proof of Corollary \ref{t:generic-empty-S-abelian-trivial}]  
Fix $s \geq 1$. 
By Corollary \ref{p:abelian-nontrivial-intro},  
the projection $\mathcal{A} \to B$  
admits only a finite number, modulo $A(\C)$, 
of sections $\sigma \in \mathcal{A}(B)$ 
which are also sections of $(\mathcal{A} \setminus \DD)\vert_U \to U$ 
for $U= B \setminus S$ where $S\in B$ is some finite subset of cardinality  $s$.  
Let $I_s$ be the finite subset of such sections $\sigma$ 
with $\deg \sigma^* \DD = s$. 
Since $\DD(K)=\varnothing$, 
Lemma \ref{l:degree-translate-abelian} implies that we have a well-defined morphism 
$f_\sigma \colon A \to B^{(s)}$, $a \mapsto (a + \sigma)^*\DD$. 
Here, $a+ \sigma$ denotes simply the composition of $\sigma$ by the translation-by-$a$ map. 
Define $V_s = \cup_{\sigma \in I_s} f_\sigma(A)$. 
By the definition of the morphisms $f_{\sigma}$, 
the image  in $B^{(s)}$  of every  finite subset $S \subset B$ verifying the condition (ii) must be contained in $V_s$. 
On the other hand, since the morphisms $f_\sigma$'s are proper, 
$V_s$ is a closed algebraic subset of $B^{(s)}$. 
Clearly, $\dim V_s \leq \dim A$ and the conclusion follows.  
 \end{proof}

\bibliographystyle{siam}

\begin{thebibliography}{10}
 

\bibitem{arakelov-71}
{\sc S.J.~Arakelov}, 
{\em Families of algebraic curves with fixed degeneracies}, 
Izv. Akad.Nauk. SSSR Ser. Mat. 35(6) (1971), 
pp.~1277--1302. 
 
 

\bibitem{raynaud-neron-model-book}
{\sc S.~Bosch, W.~L\" utkebohmert, and M.~Raynaud}, 
{\em N\' eron models},  Ergebnisse der Mathematik und ihrer Grenzgebiete (3) [Results in Mathematics and Related Areas (3)], 
vol. 21, Springer-Verlag, Berlin, (1990). 
MR1045822 (91i:14034)
 
 

\bibitem{buium-94}
{\sc A.~Buium}, 
{\em The abc Theorem for Abelian Varieties}, 
International Mathematics Research Notices, No. 5 (1994), 
pp.~219--233. 

 
 



\bibitem{capo-harris-mazur}
{\sc L.~Caporaso, J.~Harris and B.~Mazur}, 
{\em Uniformity of rational points}, 
Journal of American Mathematics Society, Vol. 7, N. 3, January 97, pp.~1--33.


\bibitem{capo-99}
{\sc L.~Caporaso}, 
{\em On certain uniformity properties of curves over function fields}, 
Compositio Mathematica, 130 (2002), pp.~1--19.  
 

\bibitem{chow-image-trace-1}
{\sc L.~Chow}, 
{\em Abelian varieties over function fields}, 
Trans. AMS 78 (1955), pp.~253--275.



\bibitem{chow-image-trace-2}
{\sc L.~Chow}, 
{\em On abelian varieties over function fields},  
Proc. Natl. Acad. Sci. USA 41 (1955), pp.~582--586. 


\bibitem{demarco-mavraki}
{\sc L.~DeMarco and N.M.~Mavraki}, 
{\em Variation of canonical height and equidistribution}, arXiv:1701.07947v2, 
(to appear in American J. Math.)


 
 
     
\bibitem{fulton-98}
{\sc W. Fulton}, 
{\em Intersection Theory}, 
Second Edition, Springer-Verlag New York (1998), 
DOI: 10.1007/978-1-4612-1700-8.  

 
 
 
   
\bibitem{ega-4-3}
{\sc A.~Grothendieck}, 
{\em \'{E}l\'ements de g\'eom\'etrie alg\'ebrique. {IV}. 
\'{E}tude locale des sch\'emas et des morphismes de sch\'emas. {III}}, 
Inst. Hautes \'Etudes Sci. Publ. Math., (1966). 

 



\bibitem{fga-4-hilbert} 
{\sc A.~Grothendieck}, 
{\em Fondements de la G\' eom\' etrie Alg\' ebrique}, 
Extraits du Seminarie Bourbaki (1957-62). Lecture Notes in Math. vol. 224, Springer, 1971. 

 
 

\bibitem{hindry-silverman-88}
{\sc M.~Hindry and  J.H.~Silverman}, 
{\em The canonical height and integral points on elliptic curves}, 
Inventiones mathematicae, vol. 93 (1988), pp.~419--450. 


\bibitem{hindry-silverman-book}
{\sc M.~Hindry and  J.H.~Silverman}, 
{\em Diophantine Geometry: An Introduction}, 
Graduate texts in mathematics : 201, Springer-Verlag Newyork (2000). 
 
 
  

\bibitem{lang-neron-theorem-paper}
{\sc S.~Lang, A.~N\' eron}, 
{\em Rational points of abelian varieties over function fields}, 
Amer. J. Math., 81 (1959), pp.~95--118.

 
 

\bibitem{lazarsfeld-positivity-I}
{\sc R.~Lazarsfeld}, 
{\em Positivity in Algebraic
Geometry, I. Classical Setting: Line Bundles and Linear Series}, 
Ergebnisse der Mathematik und ihrer Grenzgebiete, 
vol. 48, Springer 2004. 
doi: 10.1007/978-3-642-18808-4 



\bibitem{liu-alg-geom}
{\sc Q.~Liu}, 
{\em Algebraic geometry and arithmetic curves}, 
vol.~6 of Oxford Graduate Texts in Mathematics, Oxford University Press, Oxford, 2002. 
\newblock Translated from the French by Reinie Ern\'e, Oxford Science Publications.
 
     
\bibitem{milne-JVs}
{\sc J.S.~Milne}, 
{\em Jacobian Varieties}, 
Arithmetic Geometry, Springer New York, (1986), pp.~167--212, doi:10.1007/978-1-4613-8655-1-7. 
 
  

\bibitem{mordell}
{\sc L.J.~Mordell}, 
{\em On the rational solutions of the indeterminate equation of the third and fourth degrees}, 
Proc. Cambridge Philos. Soc. 21 (1922), pp.~179--192.
 
 \bibitem{mumford-GIT}
{\sc D.~Mumford}, 
{\em Geometric invariant theory}, 
Ergebnisse der Math., Neue F. Bd. 34,
Springer-Verlag, 1965.

 
\bibitem{neeman}
{\sc A.~Neeman}, {\em Algebraic and analytic geometry},
London Math. Soc. Lec. Note Series 345 (2007).  

\bibitem{noguchi-winkelmann-04}
{\sc J.~Noguchi and J.~Winkelmann}, 
{\em Bounds for curves in abelian varieties}, 
Journal f\" ur die reine und angewandte Mathematik vol. 572 (2004), 
pp.~27--47. 



\bibitem{parshin-68}
{\sc A.N.~Parshin}, 
{\em Algebraic curves over function fields I}, 
Izv. Akad. Nauk. SSSR Ser. Math. 32 (1968), pp.~1191--1219.


 

\bibitem{parshin-90}
{\sc A.N.~Parshin}, 
{\em Finiteness Theorems and Hyperbolic Manifolds}, 
The Grothendieck Festschrift: A Collection of Articles Written in Honor of the 60th Birthday of Alexander Grothendieck, 
Birkh{\"a}user Boston, Boston, MA, (1990), pp.~163--78. 
doi: 10.1007/978-0-8176-4576-2-6. 



\bibitem{phung-19-elliptic}
{\sc X.K.~Phung}, 
{\em Large unions of generalized integral sections on elliptic surfaces},   
Preprint, 2019. 


\bibitem{phung-19-abelian}
{\sc X.K.~Phung}, 
{\em Generalized integral points on abelian varieties and the Geometric Lang-Vojta conjecture},   
Preprint, 2019. 

\bibitem{phung-19-parshin-integral}
{\sc X.K.~Phung}, 
{\em On Parshin-Arakelov theorem and uniformity of $S$-integral sections on elliptic surfaces},   
Preprint, 2019. 
 

 
 
 
 
 \bibitem{silverman-lenstra-counting-lemma}
{\sc J.H.~Silverman}, 
{\em Integral points and the rank of Thue elliptic curves}, 
Inventiones mathematicae vol. 66 (1982), pp.~395--404.  

  

 
\bibitem{stack-project}
{\sc Stack Projects Authors}, 
{\em Stacks Project}, 
\text{stacks.math.columbia.edu}. 
  
     
\bibitem{zannier-unlikely-intersection}  
{\sc U.~Zannier}, 
{\em Some problems of Unlikely Intersections in Arithmetic and Geometry}, 
with Appendixes by David Masser, Annals of mathematics studies, no. 181, Priceton University Press (2012). 


\end{thebibliography}

\end{document}